\documentclass[review]{elsarticle}

\usepackage{amsthm}
\usepackage{amsmath}
\usepackage{amssymb}
\usepackage{txfonts}
\usepackage{tabls}
\usepackage{indentfirst}
\usepackage{pgf,tikz}
\usepackage{array}
\usepackage{enumerate}
\usepackage{cases}
 \usepackage[arrow,matrix]{xy}
 \usepackage{color,xcolor}
 \newcolumntype{C}[1]{>{\centering\arraybackslash}p{#1}}

\newtheorem{thm}{Theorem}[section]
 
 \newtheorem{lem}[thm]{Lemma}
 \newtheorem{prop}[thm]{Proposition}
 \newtheorem{ex}[thm]{Example}

 \theoremstyle{definition}
  \newtheorem{defn}[thm]{Definition}
 \theoremstyle{remark}
 \newtheorem{rem}[thm]{Remark}

\journal{Journal of Functional Analysis}

\begin{document}

\begin{frontmatter}

\title{Large scale properties for bounded automata groups\tnoteref{mytitlenote}}
\tnotetext[mytitlenote]{The authors were partially supported by the key discipline innovative talent training plan of Fudan University(EZH1411370/003/005)}

\author{Xiaoman Chen}
\ead{xchen@fudan.edu.cn}

\author{Jiawen Zhang\corref{mycorrespondingauthor}}
\cortext[mycorrespondingauthor]{Corresponding author}
\ead{jiawenzhang12@fudan.edu.cn}

\address{School of Mathematical Sciences, Fudan University, 220 Handan Road, Shanghai, 200433, China}

\begin{abstract}
In this paper, we study some large scale properties of the mother groups of bounded automata groups.
First we give two methods to prove every mother group has infinite asymptotic dimension. Then we study the decomposition complexity of certain subgroup in the mother group. We prove the subgroup belongs to $\mathcal{D}_\omega$.
\end{abstract}

\begin{keyword}
Automata group (self similar group); bounded automata group; asymptotic dimension; finite decomposition complexity
\end{keyword}

\end{frontmatter}


\section{Introduction}

Self similar groups (groups generated by automata) were introduced by V. M. Glu\v{s}hkov \cite{Glu61} in the 1960s, and are now very important in different aspects of mathematics. They are generated by simple automata, but their structures are very complicated and they possess a lot of interesting properties which are hard to find in classical ways. These properties help to answer some famous problems in the early times. For example, the Grigorchuk group \cite{Gri06} can be defined by an automaton with five states over two letters. It is the first example of a group with intermediate growth \cite{Gri85}, which answered the Milnor problem, and it is also a finitely generated infinite torsion group \cite{Gri80}, which answered one of the Burnside problems.

The class of bounded automata groups is a special kind of self similar groups with relatively simple structures, which has been first defined and studied by S. Sidki (\cite{Sid00}, \cite{Sid04}). This class is very large and it contains most of the well-studied groups, like the Grigorchuk group, the Gupta-Sidki group \cite{GS83}, the Basilica group and so on. S. Sidki proved the structure theorem of bounded automata groups in \cite{Sid00}, which describes how elements in them look like. Recently an embedding theorem has been proven \cite{BKN10} which said that there exists a series of mother groups such that every finitely generated bounded automata group can be embedded into one of them. And it has also been proven that mother groups are amenable, so is any bounded group.

In this paper, we study two large scale properties of the mother groups: asymptotic dimension and finite decomposition complexity. Asymptotic dimension was firstly introduced by Gromov in 1993 as a coarse analogue of the classical topological covering dimension, but it didn't get much attention until G. Yu in 1998 proved that the Novikov higher signature conjecture holds for groups with finite asymptotic dimensions \cite{Yu98}. So it is important to study whether the mother groups have finite asymptotic dimensions or not. In \cite{Smi07} J. Smith has proved that the Grigorchuk group has infinite asymptotic dimension, then by the embedding theorem, most of the mother groups have infinite asymptotic dimensions, except several ones with fewer letters. We prove:

\noindent{\bf Main Theorem 1.}
All of the mother groups $G_d$ of bounded automata groups have infinite asymptotic dimensions for $d>2$.

We prove this theorem by two different methods. One is to show the mother group $G_3$ is coarsely equivalent to the cubic power of itself. Another is more precise: we show that the direct sum of countable infinitely many copies of integer can be embedded into all of the mother groups $G_d$ for $d>2$.

Next, we study the decomposition complexity of the mother group $G_3$. Finite decomposition complexity (FDC) is a concept introduced by E. Guentner, R. Tessera and G. Yu \cite{GTY12} in order to solve certain strong rigidity problem including the stable Borel conjecture. It generalizes finite asymptotic dimension. Briefly speaking, a metric space has FDC if it admits an algorithm to decompose itself into some nice pieces which are easy to handle in certain asymptotic way. We focus on the decomposition complexity of a special subgroup in the mother group $G_3$. It was derived naturally from the proof of the first main theorem. We study the commutative relations between the generators, then use induction to prove this subgroup belongs to $\mathcal{D}_\omega$ where $\omega$ is the first infinite ordinal number. In particular, this subgroup has FDC. The notion $\mathcal{D}_\omega$ will be introduced in the next section.

\noindent{\bf Main Theorem 2.}
The mother group $G_3$ contains a subgroup $T$ which belongs to $\mathcal{D}_\omega$ but has infinite asymptotic dimension.

The paper is organized as follows. In Section 2, we recall some basic definitions and properties of automata group, asymptotic dimension, and finite decomposition complexity. In Section 3, we recall the mother groups and the embedding theorem for bounded automata groups. Then we prove our first main theorem. In the last section, we focus on the special subgroup in $G_3$. We prove it has finite decomposition complexity. More explicitly, it belongs to $\mathcal{D}_\omega$.

\noindent{\bf Acknowledgment.}
We thank Guoliang Yu, Yijun Yao and Andrzej Zuk for many stimulating discussions.

\section{Preliminaries}
In this section, we introduce the basic concepts of automata groups. See \cite{Nek05} for classical references on automata groups.

\subsection{Rooted tree $\mathbf{X^*}$ and its automorphism group $\mathbf{Aut(X^*)}$}
We first recall some basic notions of rooted trees and their automorphism groups. See Chapter One of \cite{Nek05} for reference.

Let $\mathrm{X}$ be a finite set with cardinality $d$, which we call \emph{alphabet}. Define $\mathrm{X^*}$ to be the set of all finite words over the alphabet $\mathrm{X}$, i.e. $\mathrm{X^*}=\{x_1x_2\cdots x_n:x_i \in \mathrm{X}, n=0,1,2,\cdots\}$. There is a natural corresponding between $\mathrm{X^*}$ and the vertices set of a rooted $d-$regular tree $T_d$ in which two words are connected by an edge if and only if they are of the form $w$ and $wx$, where $w \in \mathrm{X^*}$, $x \in \mathrm{X}$. The empty word $\emptyset$ is the root of the tree. For any finite word $v$ in $\mathrm{X^*}$, we use $|v|$ to denote the level of $v$. The set $\mathrm{X}^n$ is the $nth$ \emph{level} of $\mathrm{X^*}$.

A map $f:\mathrm{X^*}\rightarrow \mathrm{X^*}$ is called an \emph{endomorphism} of the tree $\mathrm{X^*}$ if it preserves the root and adjacency of the vertices. An \emph{automorphism} is a bijective endomorphism. Denote by $\mathrm{Aut(X^*)}$ the automorphism group of the rooted tree $\mathrm{X^*}$.
We recall the definition of the wreath product here for further explanation of $\mathrm{Aut(X^*)}$.

\begin{defn}\label{wreath product}
Let $G$ be a group, and $d$ be a positive integer. Denote by $S_d$ the permutation group of $d$ elements. There is a natural action of $S_d$ on $G^d$ defined by $\sigma \cdot (g_1,g_2,\cdots,g_d)=(g_{\sigma(1)},g_{\sigma(2)},\cdots,g_{\sigma(d)})$, where $\sigma \in S_d$, $(g_1,g_2,\cdots,g_d)\in G^d$. Define the wreath product $G\wr d$ to be the semi-product $G^d\rtimes S_d$. More Explicitly, the multiplication in $G\wr d$ is given by
$$\big((g_1,g_2,\cdots,g_d),\sigma\big)\cdot \big((h_1,h_2,\cdots,h_d),\tau\big)=\big((g_1h_{\sigma(1)},g_2h_{\sigma(2)},\cdots,g_dh_{\sigma(d)}),\sigma\tau\big).$$
\end{defn}

Let $g \in \mathrm{Aut(X^*)}$, and fix a vertex $v \in \mathrm{X^*}$. The subtree $v\mathrm{X^*}$ is the rooted tree with the root $v$ and all the words in $\mathrm{X^*}$ starting with $v$. Then $g$ naturally induces a map $v\mathrm{X^*}\rightarrow g(v)\mathrm{X^*}$. We can identify the tree $\mathrm{X^*}$ with the subtree $v\mathrm{X^*}$ by sending $w$ to $vw$, also $\mathrm{X^*}$ with $g(v)\mathrm{X^*}$. Under these identifications, $g$ induces an automorphism $g|_v \in \mathrm{Aut(X^*)}$ which we call the \emph{restriction} of $g$ on $v$.

Now we can resolve an automorphism of a rooted regular tree into several automorphisms of its subtrees as follows.

\begin{prop}\label{decomposition of automorphisms}
Let $\mathrm{X}=\{1,2,\cdots,d\}$, then there is an isomorphism
$$\psi:\mathrm{Aut(X^*)}  \rightarrow  \mathrm{Aut(X^*)}\wr d ,$$
given by
$$g \mapsto (g|_1,g|_2,\cdots,g|_d)\sigma,$$
where $\sigma$ is the action of $g$ on $\mathrm{X} \subset \mathrm{X^*}$.
\end{prop}

In the following, we will use $g=(g|_1,g|_2,\cdots,g|_d)\sigma$ to represent the above map.
We also introduce a graph to represent the above proposition as follows.
Draw the $0$th level and the $1$st level of the $d-$regular tree $T_d$. For a given element $g \in \mathrm{Aut(X^*)}$, suppose $\psi(g)=(g|_1,g|_2,\cdots,g|_d)\sigma$ where $\psi$ is defined in Proposition \ref{decomposition of automorphisms}. Label the root by $\sigma$, and the first level by $g|_1,g|_2,\cdots,g|_d$ in order from left to right. We call it the \emph{graph representation} of $g$.
We draw the case $d=3$ as an example.

$$
g=
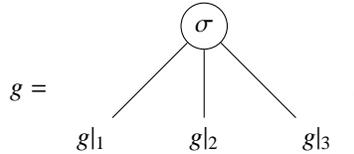
\begin{array}{c}
          \begin{tikzpicture}
            \tikzstyle{l} = [];
            \tikzstyle{r} = [draw,circle]
                \node[r] {$\sigma$} [grow'=down]
                child {
                    node[l] (t1) {$g|_3$}
                }
                child {
                    node[l] (t2) {$g|_2$}
                }
                child {
                    node[l] (t3) {$g|_1$}
                };
           \end{tikzpicture}
\end{array}.
$$

\subsection{Automata}
We will introduce another point of view on the automorphism group of a rooted regular tree.
Let $\mathrm{X}$ be as above.

\begin{defn}\label{automata}
An \emph{automaton} $\mathrm{A}$ over the alphabet $\mathrm{X}$ is given by two things,
\begin{itemize}
  \item the set of $\mathrm{state}$s, also denoted by $\mathrm{A}$;
  \item a map $\tau: \mathrm{X} \times \mathrm{A}  \rightarrow  \mathrm{X} \times \mathrm{A}$.
\end{itemize}
If $\tau(x,q)=(y,p)$, then $y$ and $p$ as functions of $(x,q)$ are called the \emph{output} and the \emph{transition function}, respectively. We denote them by $y=\mathrm{A}_\Box(x,q)$, and $p=\mathrm{A}_\bullet(x,q)$. $\mathrm{A}$ is called \emph{invertible} if $\tau(\cdot, q)$ is a bijection $\mathrm{X} \rightarrow \mathrm{X}$ for any state $q$.
\end{defn}

We interpret an invertible automaton as a machine which produces automorphisms of $\mathrm{X^*}$ as follows.
Fix a state $q$, if we input a letter $x \in \mathrm{X}$, then we have the output $y=\mathrm{A}_\Box(x,q)$ and a new state $p=\mathrm{A}_\bullet(x,q)$. Next we input a letter $z$, then we can get another output $w=\mathrm{A}_\Box(z,p)$ and another state $s=\mathrm{A}_\bullet(z,p)$. Inductively, we can define an automaton $\mathrm{A^*}$ with alphabet $\mathrm{X^*}$ and the same state space as $\mathrm{A}$ by
\begin{eqnarray*}
  \mathrm{A}_{\Box}^*(x_1x_2\cdots x_n,q)&=&\mathrm{A}_\Box(x_1,q) \mathrm{A}_\Box^* (x_2\cdots x_n, \mathrm{A}_\bullet(x_1,q)),\\
  \mathrm{A}_\bullet^*(x_1x_2\cdots x_n,q)&=&\mathrm{A}_\bullet^*(x_2\cdots x_n, \mathrm{A}_\bullet(x_1,q)).
\end{eqnarray*}
In this way the automaton with an initial state $q$ can be associated with an endomorphism $g$.
Because the automaton is invertible, $g$ is an automorphism.

\subsection{Self similar group}
We recall the definition of self similar groups.

\begin{defn}\label{self similar group}
Let $\mathrm{X}$ be a finite set with $d$ elements, and $G$ be a subgroup in $\mathrm{Aut(X^*)}$. $G$ is called \emph{self similar} if for any $v \in \mathrm{X^*}$, one has $g|_v \in G$.
\end{defn}

Recall that we have defined a group isomorphism $\psi$ in Proposition \ref{decomposition of automorphisms}. Then the above definition is equivalent to say that there exists a group homomorphism $\varphi:G\rightarrow G\wr d$, defined by the restriction of $\psi$ on $G$. The map $\varphi$ is called the \emph{wreath recursion} of $G$, and also called the \emph{self similar structure} of $G$.

Generally, suppose $G$ is any group, not necessarily a subgroup in $\mathrm{Aut(X^*)}$. Given a group homomorphism $\varphi:G\rightarrow G\wr d$, then there exists a group homomorphism $\rho: G\rightarrow \mathrm{Aut(X^*)}$. In other words, $G$ acts on the $d-$regular tree $\mathrm{X^*}$, and the image $\mathrm{Im}(\rho)$ is self similar in the above sense. In this situation, we also call $G$ a \emph{self similar group}.

Now we interpret self similar group in terms of automata. Given an automaton $\mathrm{A}$ and fix a state $q$, associate an automorphism $g$ as explained in the above subsection. For convenience, we denote such $g$ by $\mathrm{A}_q$. Let $G$ be the subgroup of $\mathrm{Aut(X^*)}$ generated by $\{\mathrm{A}_q: q\,\mbox{is a state}\}$. It's easy to see that $G$ is a self similar group in the above definition, and $G$ is called the \emph{automata group generated by $\mathrm{A}$}. Conversely, given any self similar group $G$, it's easy to construct an automaton $\mathrm{A}$ such that the associated group is just $G$. From now on, we will abuse the words "self similar group" and "automata group".

\begin{ex}(See \cite{Gri06}.)\label{Grigorchuk group}
We give a famous example of the self similar group, the Grigorchuk group, which answered a lot of problems explained in the first section.
Let $T=T_2$ be a rooted binary tree, and the Grigorchuk group $\mathbb{G}$ is a subgroup of the automorphism group $\mathrm{Aut}(T)$. $\mathbb{G}$ is generated by four elements defined recursively as follows:
$$a=(1,1)\sigma,\quad b=(a,c),\quad c=(a,d),\quad d=(1,b),$$
where $\sigma=(12) \in S_2$. Here the equal sign is in the sense of Proposition \ref{decomposition of automorphisms}.

This group is infinite, of intermediate growth, and every element has finite order.
\end{ex}

\subsection{Bounded automata groups}
We introduce the main object of this paper, the bounded automata group which was first defined and studied by S. Sidki \cite{Sid00}. We also recommend \cite{BKN10} for reference.

Let $\mathrm{X}$ be as above, and $G$ be a self similar subgroup in $\mathrm{Aut(X^*)}$ generated by an automaton $\mathrm{A}$. Given an automorphism $\alpha \in G$, define the \emph{set of states} of $\alpha$ to be
$$S(\alpha)=\{\alpha|_w: w \in \mathrm{X^*}\}.$$
If $S(\alpha)$ is finite, then $\alpha$ is called \emph{automatic}. The set of all automatic automorphisms forms a subgroup $\mathfrak{A}(\mathrm{X})$ in $\mathrm{Aut(X^*)}$.

An automorphism $\alpha$ is called \emph{bounded} if the sets $\{w \in \mathrm{X}^n: \alpha|_w\neq 1\}$ have uniformly bounded cardinalities over all $n$. The set of all bounded automorphisms forms a subgroup $\mathfrak{B}(\mathrm{X})$ in $\mathrm{Aut(X^*)}$. Denote by $\mathfrak{BA}(X)=\mathfrak{B}(\mathrm{X})\cap \mathfrak{A}(\mathrm{X})$ the group of all bounded automatic automorphisms of the regular tree $\mathrm{X^*}$. A group $G$ is called a \emph{bounded automata group} if it is a subgroup of $\mathfrak{BA}(X)$ for some $X$.

Sidki has studied the description of bounded automorphisms. To state his result, we need some more notions.

An automorphism $\alpha$ is called \emph{finitary} if there exists a non-negative integer $l$ such that for any $w\in \mathrm{X}^l$, $\alpha|_w=1$. The smallest number $l$ with such property is called the \emph{finitary depth} of $\alpha$.
An automorphism $\alpha$ is called \emph{directed} if there exists a word $w_0\in \mathrm{X}^l$ such that $\alpha|_{w_0}=\alpha$, and all the other states $\alpha|_w$ for $w \in \mathrm{X}^l$ are finitary. The smallest number $l$ with such property is called the \emph{period} of $\alpha$.

Sidki got the following theorem describing the bounded automorphisms.

\begin{thm}\label{struction of bounded automata}(See \cite{Sid00}.)
An automatic automorphism $\alpha$ is bounded if and only if it is either finitary, or there exists an integer $m$ such that all non-finitary states $\alpha|_w$ with $w \in \mathrm{X}^m$ are directed.
\end{thm}
By this theorem, we can define the depth of an automatic bounded automorphism.

\begin{defn}\label{depth}
Let $\alpha$ be an automatic bounded automorphism. Define its \emph{depth} as follows. If it is finitary, then its depth is just its finitary depth defined above. Otherwise, its depth is the smallest $m$ in Theorem \ref{struction of bounded automata}, which is also called the \emph{bounded depth}.
\end{defn}

\subsection{Asymptotic Dimension and Finite Decomposition Complexity}
In this section, we recall two conceptions in coarse geometry: asymptotic dimension and finite decomposition complexity (FDC).
Asymptotic dimension was first introduced by Gromov in 1993, but it didn't get much attention until G. Yu proved that the Novikov higher signature conjecture holds for groups with finite asymptotic dimension in 1998 \cite{Yu98}. Here we also recommend \cite{BD05} for reference. FDC is a conception which generalizes finite asymptotic dimension. It was recently introduced by E. Guentner, R. Tessera and G. Yu (\cite{GTY12}) to solve certain strong rigidity problem including the stable Borel conjecture. See also \cite{WC11}.

Let $X$ be a metric space and $r>0$. We call a family $\mathcal{U}=\{U_i\}$ of subsets in $X$ $r-$\emph{disjoint}, if for any $U\neq U'$ in $\mathcal{U}$, $d(U,U')>r$, where $d(U,U')=\mathrm{inf}\{d(x,x'):x\in U,x'\in U'\}$. We write
$$X=\bigsqcup_{r-disjoint}U_i$$
for this. We call a cover $\mathcal{V}$ \emph{uniformly bounded}, if sup$\{\mathrm{diam}(V):V\in\mathcal{V}\}$ is finite.

\begin{defn}
Let $X$ be a metric space. We say that the \emph{asymptotic dimension} of $X$ doesn't exceed $n$ and write asdim$X\leqslant n$, if for every $r>0$, the space $X$ can be covered by $n+1$ subspaces $X_0,X_1,\cdots,X_n$, and each $X_i$ can be further decomposed into some $r-$disjoint uniformly bounded subspaces:
$$X=\bigcup^n_{i=0}X_i,\mbox{\quad}X_i=\bigsqcup_{r-disjoint}X_{ij} \mbox{~~and~~}\sup_{i,j}\mbox{diam}X_{ij}<\infty.$$
We say asdim$X=n$, if asdim$X\leqslant n$ and asdim$X$ is not less than $n$.
\end{defn}
From the definition, it's easy to see that the asymptotic dimension of a subspace is not greater than that of the whole space. There are some other equivalent definitions for asymptotic dimension, but we are not going to focus on this and guide the readers to \cite{BD05} for reference. Now we introduce the notion of FDC which naturally generalizes finite asymptotic dimension.

\begin{defn}
A metric family $\mathcal{X}$ is called \emph{$r-$decomposable} over a metric family $\mathcal{Y}$ if for every $X\in\mathcal{X}$, there exists a decomposition:
$$X=X_0 \cup X_1,\mbox{\quad}X_i=\bigsqcup_{r-disjoint}X_{ij},$$
where $X_{ij}\in \mathcal{Y}$. It's denoted by $\mathcal{X}\stackrel{r}{\rightarrow}\mathcal{Y}$.
\end{defn}

\begin{defn}
(See \cite{GTY12}.)
\begin{itemize}
  \item Let $\mathcal{D}_0$ be the collection of all the bounded families.
  \item For any ordinal number $\alpha>0$, define:
        $$\mathcal{D}_\alpha=\{\mathcal{X}:\forall r>0,\exists\beta<\alpha, \exists \mathcal{Y}\in\mathcal{D}_\beta, \mbox{~such that~}\mathcal{X}\stackrel{r}{\rightarrow}\mathcal{Y}\}.$$
\end{itemize}
We call a metric family $\mathcal{X}$ has finite decomposition complexity (FDC) if there exists some ordinal number $\alpha$ such that $X$ is in $\mathcal{D}_\alpha$. There are other equivalent definitions for FDC, we recommend \cite{GTY12} for reference. We say a single metric space $X$ has FDC if $\{X\}$, viewed as a metric family, has FDC. In \cite{GTY12}, we know that $X$ has finite asymptotic dimension if and only if there exists a non-negative integer $n$, such that $X\in \mathcal{D}_n$.
\end{defn}

Next, we introduce some coarse permanence properties of asymptotic dimension and FDC. We state the following properties in the case that the metric family consists of only one metric space.
First let's recall some basic definitions in coarse geometry \cite{NY12}.
Let $X,Y$ be two metric spaces, and $f:X\rightarrow Y$ be a map.
\begin{itemize}
  \item $f$ is called \emph{bornologous} if there exists a non-decreasing proper function $\rho_1: \mathbb{R^+}\rightarrow \mathbb{R}$ such that for every $x,x'\in X$,
      $$d_Y(f(x),f(x'))\leqslant \rho_1(d_X(x,x'));$$
  \item $f$ is called \emph{effectively proper} if there exists a non-decreasing proper function $\rho_2: \mathbb{R^+}\rightarrow \mathbb{R}$ such that for every $x,x'\in X$,
      $$\rho_2(d_X(x,x'))\leqslant d_Y(f(x),f(x'));$$
  \item $f$ is called a \emph{coarse embedding}, if $f$ is both bornologous and effectively proper.
\end{itemize}
$X$ and $Y$ are called \emph{coarsely equivalent} if there exists a coarse embedding $f:X\rightarrow Y$ and $f(X)$ is a net in $Y$, i.e. there exists some constant $R>0$, such that for any $y\in Y$, there exists some $x\in X$ satisfying $d(f(x),y)<R$. Asymptotic dimension and FDC are coarse invariants. More explicitly, we have the following proposition.

\begin{prop}\label{coarse invariance of asdim}
Suppose two metric spaces $X$ and $Y$ are coarsely equivalent, then asdim$X$ = asdim$Y$; $X$ has FDC if and only if $Y$ has FDC.
\end{prop}

We have the following proposition for the subspace case.
\begin{prop}\label{subspace for asdim}
If $X$ is a subset of some metric space $Y$ equipped with the induced metric, then $asdimX\leqslant asdimY$; And if $Y$ has FDC, so does $X$.
\end{prop}

Now we turn to the case of groups. Suppose $G$ is a finitely generated group with a finite generating set $\Sigma$ which is symmetric in the sense that if $\sigma\in \Sigma$, then $\sigma^{-1}\in \Sigma$. $G$ can be equipped with a word length function $l$:
$$l(g)=\mbox{min} \big\{~ n ~\big|~ g=\sigma_1\sigma_2\cdots\sigma_n,n\in \mathbb{N},\sigma_i\in\Sigma   ~\big\}.$$
Then the word length metric is induced by the formula $d(g,h)=l(gh^{-1})$.
It can be shown that for any two finite generating sets, the induced word length metric are coarsely equivalent.

The word length metric induced by a finite generating set is proper in the sense that every ball with finite radius has finitely many elements. Furthermore, it can be shown that given two proper length functions on a group $G$, the two induced length metrics are coarsely equivalent. So we can use any proper length function on the group.

\begin{prop}\label{extension}
Let $G,H$ be two groups with FDC, and let $K$ be an extension of $G$ by $H$, i.e. there exists some short exact sequence:
$1\rightarrow G \rightarrow K\rightarrow H\rightarrow 1$, then $K$ also has FDC. In particular, let $H$ be a normal subgroup of $G$, and suppose $H$ and $G/H$ have FDC, then $G$ also has FDC. More precisely, if $H\in\mathcal{D}_\alpha$ and $G/H\in\mathcal{D}_\beta$, then $G\in\mathcal{D}_{\beta+\alpha}$.
\end{prop}

\begin{ex}
Let $\mathbb{Z}$ be the integer number, then:
\begin{enumerate}[1)]
  \item asdim$(\mathbb{Z}^n)$ = $n$ for all $n\in \mathbb{N}$;
  \item $\bigoplus \mathbb{Z}$ (countable infinite direct sum) $\in \mathcal{D}_\omega$, where $\omega$ is the smallest infinite ordinal number.
\end{enumerate}
\end{ex}

\section{Bounded automata group and its mother group}

In this section we introduce our main object, a series of universal bounded automata groups in the sense that every finitely generated bounded automata group can be embedded into some wreath product of one of them.

\subsection{The Mother Group}
\begin{defn}\label{mother group}(See \cite{BKN10}.)
Let $S_d$ be the permutation group of $d$ elements, and $B_d=S_d \wr (d-1)=S_d^{d-1}\rtimes S_{d-1}$, $F_d=S_d\ast B_d$ be the free product of $S_d$ and $B_d$. Define the self similar  structure on $F_d$ recursively as
$$S_d \ni a \mapsto (1,\cdots, 1)a, \mbox{\ and\ } B_d \ni b=(b_1,\cdots, b_{d-1})\sigma \mapsto (b_1,\cdots, b_{d-1},b)\sigma.$$
Then $F_d$ is a self similar group, and there is a natural homomorphism from $F_d$ to $\mathrm{Aut}(T_d)$ explained in section 2. Define $G_d$ to be the image of $F_d$ in $\mathrm{Aut}(T_d)$, and we call $G_d$ the \emph{mother group} of degree $d$.
\end{defn}

It's easy to see that $G_d$ contains two subgroups $S_d$ and $B_d$, and it is finitely generated by $S_d\cup B_d$ for every $d$, and we will fix this special generating set in our discussion of the word length metric on $G_d$.

First let's analyse the structure of $G_2$. It is a subgroup in the automorphism group of the rooted binary tree generated by two recursively defined automorphisms
$$a=(1,1)\sigma, \quad b=(\sigma,b),$$
where $\sigma=(12) \in S_2$. By induction, this group is just the free product of the group having two elements with itself, i.e. $G_2=\mathbb{Z}_2\ast\mathbb{Z}_2$.

In \cite{BKN10}, an embedding theorem for finitely generated bounded automata groups has been proven as follows.

\begin{thm}\label{embedding theorem}(See \cite{BKN10}.)
Any finitely generated subgroup $G$ of $\mathfrak{B}\mathfrak{A}\mathrm(X)$ can be embedded as a subgroup into the wreath product $G_{d^n}\wr d^n$ for some integer $n$, where $d$ is the cardinality of $\mathrm{X}$.
\end{thm}

\begin{proof}
Suppose $G$ is generated by a finite set $S$. Let $Q=\bigcup\limits_{\alpha\in S} S(\alpha)$ be all states in the generators $S$, and $F$ be the set of finitary elements in $Q$. Let $m$ be an integer greater than the depths of all elements in $Q$, and $l$ be a common multiple of the periods of directed states in $Q$.

First, let's \texttt{kill} the finitary elements with depth greater than $1$ in $S$. Let $R=\{q|_\omega:q\in Q, \omega \in \mathrm{X}^m\}$, and $H=\langle R \rangle$ be the subgroup in $G$ generated by $R$. There is a natural embedding by Proposition \ref{decomposition of automorphisms} $m$ times:
$$G \hookrightarrow H\wr d\wr \cdots \wr d,$$
where there are $m$ times wreath products.

Next we change the alphabet to make the periods of elements in $Q$ to be 1. Replace $\mathrm{X}$ by $\mathrm{X}'=\mathrm{X}^l$, and let $T=\mathrm{X}^*$ and $T'=(\mathrm{X}')^*$. It is convenient to regard $T'$ as a subtree of $T$ consisting of all the levels which are multiples of $l$. $H$ can be viewed as a group of automatic automorphisms of $T'$. Fix a letter $o' \in \mathrm{X}'$ and a transitive cycle $\zeta \in S_d$. For any $x \in \mathrm{X}'$, put $\zeta_x=\zeta^i$ for the unique $i$ mod $|\mathrm{X}'|$ such that $x=\zeta^i(o')$. Define $\delta \in \mathrm{Aut}(T')$ by $\delta=(\delta'_x)_{x\in \mathrm{X}'}=(\delta \zeta_x^{-1})_{x\in \mathrm{X}'}$, i.e.
$$\delta: \zeta^{i_1}(o')\zeta^{i_2}(o')\cdots \zeta^{i_n}(o') \mapsto \zeta^{i_1}(o')\zeta^{i_2-i_1}(o')\cdots \zeta^{i_n-i_{n-1}}(o').$$
For any $\alpha=(\alpha'_x)_{x\in \mathrm{X}'}\sigma \in \mathrm{Aut}(T')$, its $\delta-$conjugate is
$$\alpha^\delta=\delta^{-1}\alpha\delta=\Big(\delta_x'^{-1} \alpha'_x \delta'_{\sigma(x)}\Big)_{x\in \mathrm{X}'}\sigma=\Big(\zeta_x \alpha_x'^\delta \zeta_{\sigma(x)}^{-1}\Big)_{x\in \mathrm{X}'}\sigma.$$
Each $\alpha \in R$ either belongs to $F$ or has the property that
$$
\alpha'_z
    \left\{
        \begin{array}{l}
            =\alpha \mbox{~~for precisely one letter~~} z \in \mathrm{X}',\\
            \mbox{is finitary whenever~~} x\neq z.
        \end{array}
    \right.
$$
In the latter case, consider $\beta=\zeta_z \alpha^\delta \zeta_{\sigma(z)}^{-1}$, then $\beta=(\beta'_x)_{x\in \mathrm{X}'}\rho'$ with $\beta'|_{o'}=\beta$, $\beta'_x$ is finitary for any $x\in \mathrm{X}'\backslash\{o'\}$, and $\rho'=\zeta_z \sigma \zeta_{\sigma(z)}^{-1}$ satisfies $\rho'(o')=o'$. In other words, we just change the fixed letter $\sigma$ to the fixed letter $o'$.

Up till now, all of the bounded depths of elements in $H$ have been changed to 1, we only need to change the finitary elements in $H$ to have depths 1.
Let $m'$ be an integer greater than all the finitary depths of $\beta_x$ above for all $x\in \mathrm{X}'\backslash \{o'\}$ and $\alpha \in R$. Enlarge once more the alphabet $\mathrm{X}'$ to $\mathrm{X}''=(\mathrm{X}')^{m'}$, and put $o''=(o')^{m'}$. Then all of the $\beta$ as above have the decomposition $\beta=(\beta''_x)_{x\in \mathrm{X}''}\rho''$ with $\rho(o'')=o''$ and $\beta''_{o''}=\beta$, and all of the other automorphisms $\beta''_x$ for $x\in \mathrm{X}''\backslash \{o''\}$ are finitary with depth at most 1 with respect to the new alphabet $\mathrm{X}''$. Therefore $\beta$ belongs to $G_{|\mathrm{X}''|}$. Note that $\zeta\in G_{|\mathrm{X}''|}$, so the $\delta-$conjugate $\alpha^\delta$ belongs to $G_{|\mathrm{X}''|}$, which implies the $\delta-$conjugate of $H$ is a subgroup in $G_{|\mathrm{X}''|}$.
\end{proof}

By the above theorem, it is important to study the property of the mother groups. Here we just mention a simple fact of the mother groups. It's obvious so we only give a sketch of the proof.
\begin{lem}\label{mother subgroup}
There is a natural embedding of $G_d$ into $G_{d+1}$ for all $d\geqslant 2$.
\end{lem}

\noindent {\bf Sketch of proof of Lemma \ref{mother subgroup}.}
There is an embedding of $S_d$ into $S_{d+1}$ which is induced by the embedding of $\{1,2,\ldots,d\}$ into $\{1,2,\ldots,d+1\}$ given by $k\mapsto k+1$. Recall that $G_d$ is generated by $S_d\cup B_d$, so the above induces an embedding of $S_d$ and $B_d$ into $G_{d+1}$, which can also induce the required embedding $G_d$ into $G_{d+1}$.

\subsection{ Asymptotic Dimension of the Mother Group}
It has been proven in \cite{Smi07} that the Grigorchuk group $\mathbb{G}$ has infinite asymptotic dimension, and from the above embedding theorem \ref{embedding theorem}, we know that there exists an integer $d>0$ such that $\mathbb{G}$ can be viewed as a subgroup in $G_{d^n}\wr d^n$, where $G_{d^n}$ is one of the mother groups. First we want to get an explicit $d$ with such property.
By the method in the proof of the embedding theorem, $\mathbb{G}$ can be embedded into $G_{2^3}\wr 2^3$. So the mother group $G_{2^3}$ has infinite asymptotic dimension, and from Lemma \ref{mother subgroup}, we know that for any integer $d\geqslant2^3$, $G_d$ has infinite asymptotic dimension.

We can prove a stronger theorem that all of the mother groups $G_d$ for $d>2$ have infinite asymptotic dimensions. This is our first main theorem as follows.
\begin{thm}\label{asymptotic dimension of the mother groups}
For any $d>2$, $G_d$ has infinite asymptotic dimension.
\end{thm}

We only need to prove the case of $d=3$, then the theorem can be implied by Lemma \ref{mother subgroup}. We prove the above theorem in two different ways. First let's recall the commeasurability of two groups.

\begin{defn}\label{def of commeasurable}
Two groups $G$ and $H$ are called \emph{commeasurable}, denoted by $G\approx H$, if they contain isomorphic subgroups of finite index:
$$G'\subset G,H'\subset H, G'\simeq H', \mbox{\,and\,} [G:G'],[H:H']<\infty.$$
\end{defn}

\begin{prop}\label{commesurable}
The mother group $G_3$ and $G_3\times G_3\times G_3$ are commeasurable: $G_3 \approx G_3\times G_3\times G_3$.
\end{prop}

\begin{proof}
Let $\mathbb{H}=\mathrm{Stab}(1)=\{g\in G_3 ~|~ g(v)=v \mbox{~for all~} v \mbox{~in level~}1\}$, i.e. $\mathbb{H}$ is the subgroup in $G_3$ such that every element acts trivially on the first level of the 3 rooted regular tree $T_3$. Let $\psi : G_3\rightarrow G_3\wr 3$ be the self similar structure described as above, and $\mathrm{pr}_i:G_3\times G_3\times G_3\rightarrow G_3$ be the projection onto the $i$ th position, where $i=1,2,3$. It is straightforward to check that for any $i=1,2,3$, $\mathrm{pr}_i\circ \psi (\mathbb{H})=G_3$.

Let $A_3$ be the group of all permutations of $3$ elements with even signs, i.e. $A_3=\{1,(123),(132)\}$. Let $\mathbb{B}=\langle A_3\rangle^{G_3}\lhd G_3$ be the normalizer of $A_3$ in $G_3$. We recall here $G_3$ contains $S_3$ as a subgroup, so $G_3$ also contains $A_3$ as a subgroup. Notice $G_3$ can be generated by $A_3\cup \{(12)\}\cup S_3\wr 2$, so the index $[G_3:\mathbb{B}]$ is less than or equal to the cardinality of the subgroup in $G_3$ generated by $\{(12)\}\cup S_3\wr 2$, which is a finite number. In fact, the subgroup generated by $\{(12)\}\cup S_3\wr 2$ is contained in
$$\{(\sigma_1,\sigma_2,g)\tau : \sigma_1,\sigma_2 \in S_3, g \in S_3\wr 2, \mbox{\,and\,} \tau=1 \mbox{\, or \,}(12)\},$$
which is a finite subgroup in $G_3$. So $\mathbb{B}$ has finite index in $G_3$.

Next we show $\psi(\mathbb{H})\supseteq \mathbb{B}\times 1\times 1$, where $1$ is the trivial subgroup.
First, for any $\omega \in A_3$, we want to find an element $g \in \mathbb{H}$ such that $\psi(g)=(\omega,1,1)$.
Assume $g$ has the following form
$$g=(\sigma_1,\sigma_2,h)1 \cdot (12) \cdot (\sigma_1',\sigma_2',h')1 \cdot (12) \cdot (\sigma_1'',\sigma_2'',h'')1 \cdot (12) \cdot (\sigma_1''',\sigma_2''',h''')1 \cdot (12),$$
where $h=(\sigma_1,\sigma_2)1$, $h'=(\sigma_1',\sigma_2')1$, $h''=(\sigma_1'',\sigma_2'')1$, and $h'''=(\sigma_1''',\sigma_2''')1$ are in $S_3\wr 2$. Then
$$g=(\sigma_1 \sigma_2' \sigma_1'' \sigma_2''', \sigma_2 \sigma_1' \sigma_2'' \sigma_1''', hh'h''h''').$$
To satisfy $\psi(g)=(\omega,1,1)$, it suffices to satisfy
$$
  \left\{
    \begin{array}{l}
        \sigma_1 \sigma_2' \sigma_1'' \sigma_2'''  =  \omega , \\
        \sigma_2 \sigma_1' \sigma_2'' \sigma_1'''  =  1, \\
        hh'h''h'''   =  1.
    \end{array}
  \right.
$$
While the last equation $hh'h''h'''=1$ is equivalent to $\sigma_1 \sigma_1' \sigma_1'' \sigma_1'''=1$ and $\sigma_2 \sigma_2' \sigma_2'' \sigma_2'''=1$. So we get the condition
\begin{numcases}{}
  \sigma_1 \sigma_2' \sigma_1'' \sigma_2'''  =  \omega , \label{EQ_1} \\
  \sigma_2 \sigma_1' \sigma_2'' \sigma_1'''  =  1, \label{EQ 2}\\
  \sigma_1 \sigma_1' \sigma_1'' \sigma_1'''  =  1, \label{EQ_3}\\
  \sigma_2 \sigma_2' \sigma_2'' \sigma_2'''  =  1. \label{EQ 4}
\end{numcases}
From (\ref{EQ_3}) and (\ref{EQ 4}), we have
\begin{numcases}{}
  \sigma_1'''=\sigma_1''^{-1}\sigma_1'^{-1}\sigma_1^{-1} , \label{EQ_5} \\
  \sigma_2'''=\sigma_2''^{-1}\sigma_2'^{-1}\sigma_2^{-1} . \label{EQ 6}
\end{numcases}
Combining them with (\ref{EQ_1}) and (\ref{EQ 2}), we have
\begin{eqnarray*}
   \sigma_1^{} \sigma_2' \sigma_1'' \sigma_2''^{-1} \sigma_2'^{-1} \sigma_1' \sigma_2'' \sigma_1''^{-1} \sigma_1'^{-1} \sigma_1^{-1}=\omega.
\end{eqnarray*}
Equivalently,
\begin{eqnarray*}
   \sigma_2' (\sigma_1'' \sigma_2''^{-1}) \sigma_2'^{-1} \cdot \sigma_1' (\sigma_1'' \sigma_2''^{-1})^{-1} \sigma_1'^{-1}= \sigma_1^{-1} \omega  \sigma_1^{}.
\end{eqnarray*}
Let $a=\sigma_1'' \sigma_2''^{-1}$, then
\begin{eqnarray}\label{EQ_7}
   (\sigma_2' a \sigma_2'^{-1}) \cdot (\sigma_1' a^{-1} \sigma_1'^{-1})= \sigma_1^{-1} \omega  \sigma_1^{}.
\end{eqnarray}
Notice for $\omega=(123)$ or $\omega=(132)$, we can solve the above formula by
\begin{eqnarray*}
(12)(12)(12)\cdot (23)(12)(23)=(123),\\
(12)(12)(12)\cdot (13)(12)(13)=(132).
\end{eqnarray*}
So for any $\omega \in A_3$, Equation (\ref{EQ_7}) always has a solution. In other words, for any $\omega \in A_3$, there exists some $g \in \mathbb{H}$, such that $\psi(g)=(\omega, 1,1)$.

Because $\mathrm{pr}_1\circ \psi$ is surjective, for any $x \in G_3$, there exists $h \in \mathbb{H}$ such that $\mathrm{pr}_1\circ \psi(h)=x$, then
$$
\psi(h^{-1}gh)=(x,y,z)(\omega,1,1)(x^{-1}, y^{-1}, z^{-1})=(x\omega x^{-1},1,1).
$$
So $\psi(\mathbb{H})\supseteq \mathbb{B}\times 1\times 1$.

Similarly, $\psi(\mathbb{H})\supseteq  1\times\mathbb{B}\times 1$ and $\psi(\mathbb{H})\supseteq 1\times 1\times\mathbb{B}$. So $\psi(\mathbb{H})\supseteq \mathbb{B}\times\mathbb{B}\times\mathbb{B}$. Because $[G_3:\mathbb{B}]$ is finite, we have $[G_3\times G_3 \times G_3:\mathbb{B}\times\mathbb{B}\times\mathbb{B}]$ is also finite. So $[G_3^3:\psi(\mathbb{H})]$ is finite, which implies $G_3$ is commeasurable.

\end{proof}

\noindent {\bf Proof of Theorem \ref{asymptotic dimension of the mother groups}.}
We only need to prove $\mathrm{asdim}(G_3)=\infty$. From lemma \ref{mother subgroup}, $G_3$ contains a subgroup $G_2$, which is isomorphic to $\mathbb{Z}_2\ast \mathbb{Z}_2$. Because $G_2$ contains a subgroup which is isomorphic to the integer group $\mathbb{Z}$, so $\mathbb{Z}$ can be coarsely embedded into $G_3$. From Proposition \ref{commesurable}, we know that $G_3$ is coarsely equivalent to $G_3\times G_3\times G_3$, so $\mathbb{Z}^3$ can be coarsely embedded into $G_3$. Inductively, $\mathbb{Z}^{3n}$ can be coarsely embedded into $G_3$ for any integer $n>0$. Because $\mathrm{asdim}(\mathbb{Z}^n)=n$ and the fact that the asymptotic dimension of a space is not less than the asymptotic dimension of its subspace, we get $\mathrm{asdim}(G_3)=\infty$.
\begin{flushright}
$\Box$
\end{flushright}

\begin{rem}
In fact, every finitely generated infinite group contains an isometric copy of the integer group $\mathbb{Z}$, see for example an exercise in \cite{Har00}.
\end{rem}

\subsection{Another proof of Theorem \ref{asymptotic dimension of the mother groups}}
In this section we introduce a new method to prove Theorem 3.4. Actually, we prove that there is a subgroup in $G_3$ which is isomorphic to the direct sum of infinitely many copies of the integer number $\mathbb{Z}$.

Let $c\in G_3$ defined recursively by $c=\big(1,(23),c\big)$. In the graph representation version,
$$
c=
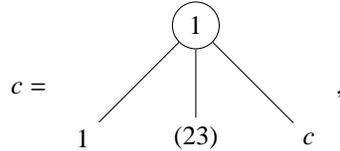
\begin{array}{c}
          \begin{tikzpicture}
            \tikzstyle{l} = [];
            \tikzstyle{r} = [draw,,circle]
                \node[r] {1} [grow'=down]
                child {
                    node[l] (t1) {$c$}
                }
                child {
                    node[l] (t2) {$(23)$}
                }
                child {
                    node[l] (t3) {1}
                };
           \end{tikzpicture}
\end{array},
$$
Here 1 means the identity map. Let $t=(23)c(23)c=(1,c(23),(23)c)$, and $K=\langle t \rangle^{G_3}$ be the normalizer of $t$ in $G_3$. In other words, $K$ is the smallest normal subgroup in $G_3$ containing $t$. Let $\mathbb{H}=\mathrm{Stab}(1)=\{g\in G_3 ~|~ g(v)=v \mbox{~for all~} v \mbox{~in level~}1\}$, i.e. $\mathbb{H}$ is the subgroup in $G_3$ such that every element acts trivially on the first level of the 3 rooted regular tree $T_3$. Let $\psi : G_3\rightarrow G_3\wr 3$ be the self similar structure described in the previous subsection. Then $K$ is a normal subgroup in $\mathbb{H}$. We have the following lemma.

\begin{lem}\label{subgroup K}
Let $K$ and $\psi$ be as above, then $K \times K \times K \leqslant \psi(K)$.
\end{lem}

\begin{proof}
First, it's straightforward to check that for any $i=1,2,3$, we have $\mathrm{pr}_i\circ \psi (\mathbb{H})=G_3$, where $\mathrm{pr}_i:G_3 \times G_3 \times G_3 \rightarrow G_3$ is the projection onto the $i$th position and $\mathbb{H}=\mathrm{Stab}(1)$.

For $t=(23)c(23)c$, take $\tilde{c}=(12)c(12)=\big( (23),1,c \big)$. Then $c\tilde{c}=\tilde{c}c$,
and $t \cdot (\tilde{c}t^{-1}\tilde{c})=(23)c(23)c \cdot \tilde{c} \cdot c(23)c(23) \cdot \tilde{c}=(23)c(23)\tilde{c}cc(23)c(23)\tilde{c}=\big[ (23)c(23)\tilde{c} \big]^2$, here we use the fact that $c$ and $\tilde{c}$ are commutative in the second equation.
Since $(23)c(23)\tilde{c}=\big( 1,c,(23) \big) \cdot \big( (23),1,c \big) =\big( (23),c,(23)c \big)$, so $t \cdot (\tilde{c}t^{-1}\tilde{c})=\big[ (23)c(23)\tilde{c} \big]^2=\big( 1,1,(23)c(23)c \big)=\big( 1,1,t \big)$.

For any $g \in G_3$, since $\mathrm{pr}_i\circ \psi (\mathbb{H})=G_3$, there exists $h\in \mathbb{H}$ such that $\psi(h)=(h_1,h_2,g)$ for some $h_1,h_2 \in G_3$. Then $\psi(h \cdot t\tilde{c}t^{-1}\tilde{c} \cdot h^{-1})=(1,1,gtg^{-1})$. So $\psi(K)\geqslant 1 \times 1 \times K$.
Notice that $(23)(1 \times 1 \times K)(23)=1 \times K \times 1$ and $(13)(1 \times 1 \times K)(13)=K \times 1 \times 1$, hence $\psi(K) \geqslant K \times K \times K$.
\end{proof}

Before we give the second proof of Theorem 3.4, we define a sequence of elements in $G_3$.

\begin{defn}
For each vertex $v$ of the 3 rooted regular tree $T_3$, define an element $t_v$ in $G_3$ inductively on the level of $v$ as follows.
\begin{enumerate}[(1)]
  \item $t_{\emptyset}=t=(23)c(23)c$,
  \item Suppose for any vertex $v$ with $|v| \leqslant n-1$, we have defined an element $t_v \in G_3$, where $|v|$ means the level of $v$, then
    \begin{itemize}
      \item If $w=1v$, define $t_w=(t_v,1,1)$;
      \item If $w=2v$, define $t_w=(1,t_v,1)$;
      \item If $w=3v$, define $t_w=(1,1,t_v)$.
    \end{itemize}
\end{enumerate}
\end{defn}

We draw the graph representations of the first few elements defined above.\\[2pt]

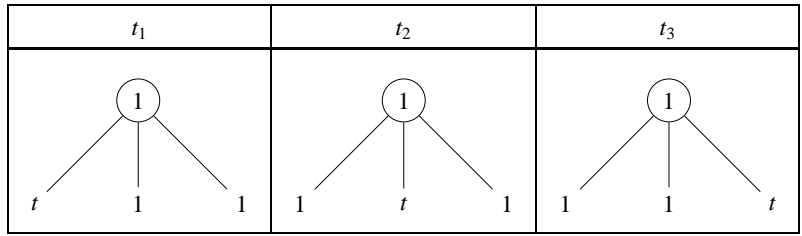
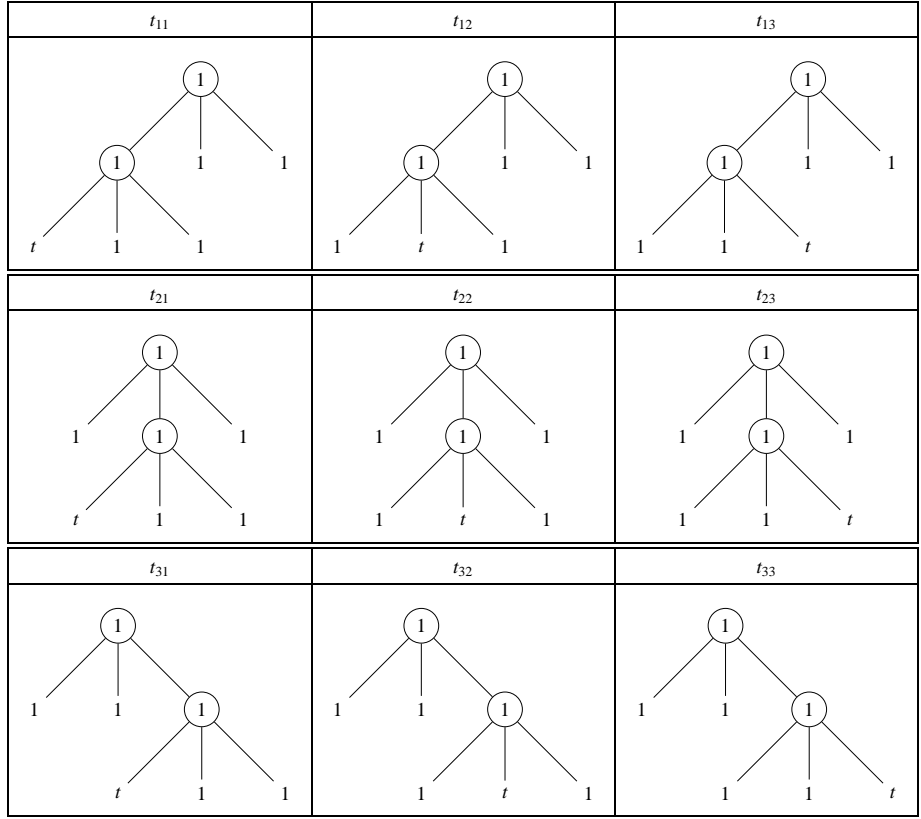
\resizebox{!}{1.6cm}{
\begin{tabular}{|c|c|c|}
  \hline
  $t_1$ & $t_2$ & $t_3$\\
  \hline
   & & \\[-0.2cm]
      \begin{tikzpicture}
            \tikzstyle{l} = [];
            \tikzstyle{r} = [draw,circle]
                \node[r] {1} [grow'=down]
                child {
                    node[l] (t1) {$1$}
                }
                child {
                    node[l] (t2) {$1$}
                }
                child {
                    node[l] (t3) {$t$}
                };
      \end{tikzpicture}
    & \begin{tikzpicture}
            \tikzstyle{l} = [];
            \tikzstyle{r} = [draw,circle]
                \node[r] {1} [grow'=down]
                child {
                    node[l] (t1) {$1$}
                }
                child {
                    node[l] (t2) {$t$}
                }
                child {
                    node[l] (t3) {$1$}
                };
       \end{tikzpicture}
    & \begin{tikzpicture}
            \tikzstyle{l} = [];
            \tikzstyle{r} = [draw,circle]
                \node[r] {1} [grow'=down]
                child {
                    node[l] (t1) {$t$}
                }
                child {
                    node[l] (t2) {$1$}
                }
                child {
                    node[l] (t3) {$1$}
                };
           \end{tikzpicture}\\
  \hline
  \end{tabular}
}

{\centering
\resizebox{!}{1.85cm}{
\begin{tabular}{|C{5cm}|C{5cm}|C{5cm}|}
  \hline
  $t_{11}$ & $t_{12}$ & $t_{13}$\\
  \hline
   & & \\[-0.2cm]
      \begin{tikzpicture}
            \tikzstyle{l} = [];
            \tikzstyle{r} = [draw,circle]
                \node[r] {1} [grow'=down]
                child {
                    node[l] (t3) {$1$}
                }
                child {
                    node[l] (t2) {$1$}
                }
                child {
                    node[r] (t1) {$1$}
                    child {
                       node[l] (t13) {$1$}
                    }
                    child {
                       node[l] (t12) {$1$}
                    }
                    child {
                       node[l] (t11) {$t$}
                    }
                };
      \end{tikzpicture}
    & \begin{tikzpicture}
            \tikzstyle{l} = [];
            \tikzstyle{r} = [draw,circle]
                \node[r] {1} [grow'=down]
                child {
                    node[l] (t3) {$1$}
                }
                child {
                    node[l] (t2) {$1$}
                }
                child {
                    node[r] (t1) {$1$}
                    child {
                       node[l] (t13) {$1$}
                    }
                    child {
                       node[l] (t12) {$t$}
                    }
                    child {
                       node[l] (t11) {$1$}
                    }
                };
      \end{tikzpicture}
    & \begin{tikzpicture}
            \tikzstyle{l} = [];
            \tikzstyle{r} = [draw,circle]
                \node[r] {1} [grow'=down]
                child {
                    node[l] (t3) {$1$}
                }
                child {
                    node[l] (t2) {$1$}
                }
                child {
                    node[r] (t1) {$1$}
                    child {
                       node[l] (t13) {$t$}
                    }
                    child {
                       node[l] (t12) {$1$}
                    }
                    child {
                       node[l] (t11) {$1$}
                    }
                };
      \end{tikzpicture}\\
  \hline
\end{tabular}
}
}

{\centering
\resizebox{!}{1.85cm}{
\begin{tabular}{|C{5cm}|C{5cm}|C{5cm}|}
  \hline
  $t_{21}$ & $t_{22}$ & $t_{23}$\\
  \hline
   & & \\[-0.2cm]
      \begin{tikzpicture}
            \tikzstyle{l} = [];
            \tikzstyle{r} = [draw,circle]
                \node[r] {1} [grow'=down]
                child {
                    node[l] (t3) {$1$}
                }
                child {
                    node[r] (t2) {$1$}
                    child {
                       node[l] (t23) {$1$}
                    }
                    child {
                       node[l] (t22) {$1$}
                    }
                    child {
                       node[l] (t21) {$t$}
                    }
                }
                child {
                    node[l] (t1) {$1$}
                };
      \end{tikzpicture}
    & \begin{tikzpicture}
            \tikzstyle{l} = [];
            \tikzstyle{r} = [draw,circle]
                \node[r] {1} [grow'=down]
                child {
                    node[l] (t3) {$1$}
                }
                child {
                    node[r] (t2) {$1$}
                    child {
                       node[l] (t23) {$1$}
                    }
                    child {
                       node[l] (t22) {$t$}
                    }
                    child {
                       node[l] (t21) {$1$}
                    }
                }
                child {
                    node[l] (t1) {$1$}
                };
      \end{tikzpicture}
    & \begin{tikzpicture}
            \tikzstyle{l} = [];
            \tikzstyle{r} = [draw,circle]
                \node[r] {1} [grow'=down]
                child {
                    node[l] (t3) {$1$}
                }
                child {
                    node[r] (t2) {$1$}
                    child {
                       node[l] (t23) {$t$}
                    }
                    child {
                       node[l] (t22) {$1$}
                    }
                    child {
                       node[l] (t21) {$1$}
                    }
                }
                child {
                    node[l] (t1) {$1$}
                };
      \end{tikzpicture}\\
  \hline
\end{tabular}
}
}

{\centering
\resizebox{!}{1.85cm}{
\begin{tabular}{|C{5cm}|C{5cm}|C{5cm}|}
  \hline
  $t_{31}$ & $t_{32}$ & $t_{33}$\\
  \hline
   & & \\[-0.2cm]
      \begin{tikzpicture}
            \tikzstyle{l} = [];
            \tikzstyle{r} = [draw,circle]
                \node[r] {1} [grow'=down]
                child {
                    node[r] (t3) {$1$}
                    child {
                       node[l] (t33) {$1$}
                    }
                    child {
                       node[l] (t32) {$1$}
                    }
                    child {
                       node[l] (t31) {$t$}
                    }
                }
                child {
                    node[l] (t2) {$1$}
                }
                child {
                    node[l] (t1) {$1$}
                };
      \end{tikzpicture}
    & \begin{tikzpicture}
            \tikzstyle{l} = [];
            \tikzstyle{r} = [draw,circle]
                \node[r] {1} [grow'=down]
                child {
                    node[r] (t3) {$1$}
                    child {
                       node[l] (t33) {$1$}
                    }
                    child {
                       node[l] (t32) {$t$}
                    }
                    child {
                       node[l] (t31) {$1$}
                    }
                }
                child {
                    node[l] (t2) {$1$}
                }
                child {
                    node[l] (t1) {$1$}
                };
      \end{tikzpicture}
    & \begin{tikzpicture}
            \tikzstyle{l} = [];
            \tikzstyle{r} = [draw,circle]
                \node[r] {1} [grow'=down]
                child {
                    node[r] (t3) {$1$}
                    child {
                       node[l] (t33) {$t$}
                    }
                    child {
                       node[l] (t32) {$1$}
                    }
                    child {
                       node[l] (t31) {$1$}
                    }
                }
                child {
                    node[l] (t2) {$1$}
                }
                child {
                    node[l] (t1) {$1$}
                };
      \end{tikzpicture}\\
  \hline
\end{tabular}
}
}

\begin{thm}\label{subgroup L}
There exists a subgroup $L$ in $G_3$ such that $L$ is isomorphic to the direct sum of infinitely many copies of integer $\mathbb{Z}$, i.e. $L \cong \bigoplus\limits_{\infty} \mathbb{Z}$.
\end{thm}

\begin{proof}
From lemma \ref{subgroup K}, $\psi(K)\geqslant K \times K \times K$. Now $t_1=(t,1,1)$ and $t\in K$, so $t_1\in K$. Similarly, $t_2,t_3\in K$. Inductively, all $t_v$ belong to $K$ for any finite word $v$.
Let $L$ be the subgroup in $G_3$ generated by the set
$$S=\big\{~t_v ~\big|~v=2,3,12,13,112,113,1112,1113\cdots,1\cdots 12,1\cdots 13,\cdots ~\big\}.$$
For a finite word $v$, the subgroup $\langle t_v \rangle$ acts trivially outside the subtree $vT_3$. So elements in $S$ are commutative with each other.

We claim that every element in $S$ generates a copy of $\mathbb{Z}$ in $G_3$. In fact, we only need to check that the subgroup generated by $t$ is isomorphic to $\mathbb{Z}$ because $\langle t_v\rangle$ is isomorphic to $\langle t \rangle$  for any finite word $v$.
From lemma \ref{mother subgroup}, $t=(23)c(23)c$ is in the image of the canonical embedding $G_2 \hookrightarrow G_3$. Notice that $G_2$ is isomorphic to $\mathbb{Z}_2 \ast \mathbb{Z}_2$ where the first copy is generated by $a=(12)$ while the second is generated by $b=\big((12),b\big)$, and $t$ is just the image of $abab$. So the subgroup generated by $t$ is isomorphic to $\mathbb{Z}$, hence it's easy to see the theorem holds.
\end{proof}

\noindent {\bf Another Proof of Theorem \ref{asymptotic dimension of the mother groups}.}
From Proposition \ref{subspace for asdim} and Theorem \ref{subgroup L}, we see that $asdimG_d\geqslant asdimG_3\geqslant asdim L=\infty$ for any $d>2$.
\begin{flushright}
$\Box$
\end{flushright}

\section{A subgroup in $G_3$ with finite decomposition complexity}
In this section, we analyse the decomposition complexity of the subgroup $T$ in $G_3$ generated by all the elements $t_v$ for any finite word $v$, i.e.
$$T=\langle t_v ~|~ v \in \{1,2,3\}^* \rangle.$$
We prove that $T$ has FDC with respect to any proper length metric.
First we need some commutative relations between elements in $J=\{t_v ~|~v \in \{1,2,3\}^*\}$. We show although they are not commutative, they satisfy certain special relations similar to commutativity. For any two finite words $v,w\in \{1,2,3\}^*$, write $v \nprec w$ if there doesn't exist some finite word $u$ such that $w=vu$. Also recall that $|v|$ denotes the level of $v$, i.e. the number of letters in $v$. For the letter 2 and 3, define $\hat{2}=3$ and $\hat{3}=2$.

\begin{prop}\label{comm relation}
For any two finite words $w_1,w_2\in \{1,2,3\}^*$, we have:
\begin{itemize}
  \item If $w_1 \nprec w_2$ and $w_2 \nprec w_1$, then $t_{w_1}\cdot t_{w_2}=t_{w_2}\cdot t_{w_1}$;
  \item If $w_2=w_1 v$, then:
    \begin{enumerate}[(1)]
      \item If the word $v$ contains 1, then $t_{w_1}\cdot t_{w_2}=t_{w_2}\cdot t_{w_1}$;
      \item Otherwise,
            \begin{enumerate}[(a)]
              \item If $|v|=1$, then $t_{w_1}\cdot t_{w_2}=t_{w_2}\cdot t_{w_1}$;
              \item If $|v|=2$, suppose $v=ab$, then $t_{w_1}\cdot t_{w_1 ab}=t_{w_1 a\hat{b}}^{-1}\cdot t_{w_1}$;
              \item If $|v|\geqslant 3$:
                    \begin{enumerate}[(i)]
                      \item $t_{w_1}\cdot t_{w_1 23a\bar{v}}=t_{w_1 22 \hat{a}\bar{v}}\cdot t_{w_1}$;
                      \item $t_{w_1}\cdot t_{w_1 32a\bar{v}}=t_{w_1 33 \hat{a}\bar{v}}\cdot t_{w_1}$;
                      \item $t_{w_1}\cdot t_{w_1 223\cdots 3}=t_{w_1 23 3\cdots 3}^{-1}\cdot t_{w_1}$;
                      \item $t_{w_1}\cdot t_{w_1 333\cdots 3}=t_{w_1 32 3\cdots 3}^{-1}\cdot t_{w_1}$;
                      \item $t_{w_1}\cdot t_{w_1 223\cdots 32}=t_{w_1 233\cdots 32}^{-1}\cdot t_{w_1}$;
                      \item $t_{w_1}\cdot t_{w_1 333\cdots 32}=t_{w_1 323\cdots 32}^{-1}\cdot t_{w_1}$;
                      \item $t_{w_1}\cdot t_{w_1 223\cdots 32a\bar{v}}=t_{w_1 233\cdots 32\hat{a} \bar{v}}\cdot t_{w_1}$;
                      \item $t_{w_1}\cdot t_{w_1 333\cdots 32a\bar{v}}=t_{w_1 323\cdots 32\hat{a} \bar{v}}\cdot t_{w_1}$,
                    \end{enumerate}
            \end{enumerate}
            where $a,b\in\{1,2,3\}$, $\bar{v}$ is a finite word and $\bar{v}$ can be $\emptyset$.
    \end{enumerate}
\end{itemize}
\end{prop}

We divide the proof into following lemmas.

\begin{lem}\label{case 1}
For any two finite words $w_1,w_2\in \{1,2,3\}^*$, if $w_1 \nprec w_2$ and $w_2 \nprec w_1$, then $t_{w_1}\cdot t_{w_2}=t_{w_2}\cdot t_{w_1}$.
\end{lem}

\begin{proof}
If $w_1 \nprec w_2$ and $w_2 \nprec w_1$, suppose $w_1=ua_1v_1$ and $w_2=ua_2v_2$ for some letters $a_1,a_2\in \{1,2,3\}$ such that $a_1\neq a_2$, and for some finite words $u,v_1,v_2\in \{1,2,3\}^*$. Suppose $a_1=2,a_2=3$. Other cases are similar.
Define a map $F:\mbox{Aut}(T_3)\rightarrow \mbox{Aut}(T_3)$ by $g \mapsto g|_u$. Then $F(t_{w_1})=t_{a_1v_1}$, $F(t_{w_2})=t_{a_2v_2}$ and $F|_{\langle t_{w_1},t_{w_2} \rangle}$ is injective. Because $t_{a_1v_1}=(1,t_{v_1},1)$ and $t_{a_2v_2}=(1,1,t_{v_2})$, so $F(t_{w_1})\cdot F(t_{w_2})=F(t_{w_2})\cdot F(t_{w_1})$, which implies that $t_{w_1}\cdot t_{w_2}=t_{w_2}\cdot t_{w_1}$.
\end{proof}

Next, we deal with the case $w_1\prec w_2$ or $w_2\prec w_1$. For convenience, we always assume that $w_1\prec w_2$, i.e. $w_2=w_1 v$ for some finite word $v$. Define a map $F:\mbox{Aut}(T_3)\rightarrow \mbox{Aut}(T_3)$ by $g \mapsto g|_{w_1}$. Then $F|_{\langle t_{w_1},t_{w_2} \rangle}$ is injective, so we only need to analyse $F(t_{w_2})=t_v$ and $F(t_{w_1})=t$.

\begin{lem}\label{case 2}
Let $v$ be a finite word and assume $v$ contains 1, then $t_v\cdot t=t\cdot t_v$.
\end{lem}

\begin{proof}
By assumption the word $v$ contains 1, i.e. $v=\tilde{u}1\tilde{v}$ for some finite words $\tilde{u}$ and $\tilde{v}$, it's easy to see that $t_v\cdot t=t\cdot t_v$ by induction on the length of $\tilde{u}$ and $\tilde{v}$.
\end{proof}

\begin{lem}\label{case 3}
Suppose $v=2$ or 3, then $t_v\cdot t=t\cdot t_v$.
\end{lem}

\begin{proof}
$$t\cdot t_2=\big( 1,c(23),(23)c \big)\big( 1,t,1 \big)=\big( 1,c(23)t,(23)c \big),$$
$$t_2\cdot t=\big( 1,t,1 \big)\big( 1,c(23),(23)c \big)=\big( 1,tc(23),(23)c \big).$$
Because $c(23)t=c(23)(23)c(23)c=(23)c=(23)c(23)cc(23)=tc(23)$, $t\cdot t_2=t_2\cdot t$. The same argument can be used to prove $t \cdot t_3=t_3 \cdot t$.
\end{proof}

\begin{lem}\label{case 4}
Let $v$ be a finite word of length 2, i.e. $v=ab$ for $a,b\in \{1,2,3\}$, then $t\cdot t_{ab}=t_{a\hat{b}}^{-1}\cdot t$.
\end{lem}

\begin{proof}
It's just a straightforward calculation. We only check the case $v=23$. Other cases are similar.
$$t \cdot t_{23} \cdot t^{-1}=\big( 1,c(23),(23)c \big) \cdot (1,t_3,1) \cdot \big( 1,(23)c,c(23) \big)=\big( 1,c(23)t_3(23)c,1 \big),$$
$$c(23)t_3(23)c=ct_2c=\big( 1,(23),c \big)\cdot (1,t,1) \cdot \big( 1,(23),c \big)=\big( 1,(23)t(23),1 \big).$$
Because $(23)t(23)=(23)\cdot (23)c(23)c \cdot (23)=t^{-1}$, $c(23)t_3(23)c=(1,t^{-1},1)=t_2^{-1}$.
So $t \cdot t_{23} \cdot t^{-1}=t_{22}^{-1}$, i.e. $t \cdot t_{23}=t_{22}^{-1} \cdot t$.
\end{proof}

Finally, we deal with the case $|v|\geqslant3$.
\begin{lem} \label{case 5}
Let $a\in {1,2,3}$ and $\bar{v}$ be a finite word, then:
\begin{itemize}
  \item $t\cdot t_{23a\bar{v}}=t_{22 \hat{a}\bar{v}}\cdot t$;
  \item $t\cdot t_{32a\bar{v}}=t_{33 \hat{a}\bar{v}}\cdot t$.
\end{itemize}
\end{lem}

\begin{proof}
We only prove the first case. The second one is similar to the first.
$$t\cdot t_{23a\bar{v}} \cdot t^{-1}=\big( 1,c(23),(23)c \big)\cdot (1,t_{3a\bar{v}},1)\cdot \big( 1,(23)c,c(23) \big)=(1,c(23)t_{3a\bar{v}}(23)c,1),$$
where
$$c(23)t_{3a\bar{v}}(23)c=ct_{2a\bar{v}}c=\big( 1,(23),c \big)\cdot (1,t_{a\bar{v}},1)\cdot \big( 1,(23),c \big)=(1,t_{\hat{a}\bar{v}},1).$$
So $t\cdot t_{23a\bar{v}} \cdot t^{-1}=t_{22 \hat{a}\bar{v}}$, in other words, $t\cdot t_{23a\bar{v}}=t_{22 \hat{a}\bar{v}}\cdot t$.
\end{proof}

\begin{lem}\label{case 6}
For $t$ and $t_v$ defined as above, we have
\begin{itemize}
  \item $t\cdot t_{223\cdots 3}=t_{233\cdots 3}^{-1}\cdot t$;
  \item $t\cdot t_{333\cdots 3}=t_{323\cdots 3}^{-1}\cdot t$,
\end{itemize}
\end{lem}

\begin{proof}
As before, we just prove the first case. The second is similar.
$$t\cdot t_{223\cdots 3}\cdot t^{-1}=\big( 1,c(23),(23)c \big)\cdot (1,t_{23\cdots 3},1)\cdot \big( 1,(23)c,c(23) \big)=(1,c(23)t_{23\cdots 3}(23)c,1).$$
Since $c(23)t_{23\cdots 3}(23)c=ct_{33\cdots3}c=(1,1,ct_{3\cdots3}c)$, we just need to calculate $ct_{3\cdots3}c$:
$$ct_{\underbrace{\scriptstyle3\cdots3}_{\scriptstyle n}}c=(1,1,ct_{\underbrace{\scriptstyle3\cdots3}_{\scriptstyle n-1}}c).$$
By induction on $n$ in the above equation, and $ctc=t^{-1}$, we have $ct_{3\cdots3}c=t_{3\cdots3}^{-1}$. So $t\cdot t_{223\cdots 3}=t_{233\cdots 3}^{-1}\cdot t$.
\end{proof}

\begin{lem}\label{case 7}
For $t$ and $t_v$ defined as above, we have
\begin{itemize}
  \item $t\cdot t_{223\cdots 32}=t_{233\cdots 32}^{-1}\cdot t$;
  \item $t\cdot t_{333\cdots 32}=t_{323\cdots 32}^{-1}\cdot t$.
\end{itemize}
\end{lem}

\begin{proof}
As before, we just prove the first case. The second is similar.
$$t\cdot t_{223\cdots 32}\cdot t^{-1}=\big( 1,c(23),(23)c \big)\cdot (1,t_{23\cdots 32},1)\cdot \big( 1,(23)c,c(23) \big)=(1,c(23)t_{23\cdots 32}(23)c,1).$$
Since $c(23)t_{23\cdots 32}(23)c=ct_{33\cdots32}c=(1,1,ct_{3\cdots32}c)$, we only need to calculate $ct_{3\cdots32}c$:
$$ct_{\underbrace{\scriptstyle3\cdots3}_{\scriptstyle n}2}c=(1,1,ct_{\underbrace{\scriptstyle3\cdots3}_{\scriptstyle n-1}2}c).$$
By induction on $n$ in the above equation, and $ct_2c=t_2^{-1}$, we have $ct_{3\cdots32}c=t_{3\cdots32}^{-1}$. So $t\cdot t_{223\cdots 32}=t_{233\cdots 32}^{-1}\cdot t$.
\end{proof}

Now we come to the last case.
\begin{lem}\label{case 8}
Let $a\in\{1,2,3\}$, $\bar{v}$ be a finite word, and $t$ be as above. Then we have:
\begin{itemize}
  \item $t\cdot t_{223\cdots 32a\bar{v}}=t_{233\cdots 32\hat{a} \bar{v}}\cdot t$;
  \item $t\cdot t_{333\cdots 32a\bar{v}}=t_{323\cdots 32\hat{a} \bar{v}}\cdot t$.
\end{itemize}
\end{lem}

\begin{proof}
As before, we just prove the first case. The second is similar.
$$t\cdot t_{223\cdots 32a\bar{v}}\cdot t^{-1}=\big( 1,c(23),(23)c \big)\cdot (1,t_{23\cdots 32a\bar{v}},1)\cdot\big( 1,(23)c,c(23) \big)=(1,c(23)t_{23\cdots 32a\bar{v}}(23)c,1).$$
Since $c(23)t_{23\cdots 32a\bar{v}}(23)c=ct_{33\cdots 32a\bar{v}}c=(1,1,ct_{3\cdots 32a\bar{v}}c)$, we only need to calculate $ct_{3\cdots 32a\bar{v}}c$:
$$ct_{\underbrace{\scriptstyle3\cdots3}_{\scriptstyle n}2a\bar{v}}c=(1,1,ct_{\underbrace{\scriptstyle3\cdots3}_{\scriptstyle n-1}2a\bar{v}}c).$$
By induction on $n$ in the above equation, it reduces to calculate $ct_{2a\bar{v}}c$.
Since $ct_{2a\bar{v}}c=\big( 1,(23),c \big)\cdot(1,t_{a\bar{v}},1)\cdot\big( 1,(23),c \big)=(1,t_{\hat{a}\bar{v}},1)$,
we have $t\cdot t_{223\cdots 32a\bar{v}}\cdot t^{-1}=t_{233\cdots 32\hat{a} \bar{v}}$. In other words, $t\cdot t_{223\cdots 32a\bar{v}}=t_{233\cdots 32\hat{a} \bar{v}}\cdot t$.
\end{proof}

\noindent {\bf Proof of Proposition \ref{comm relation}.}
It follows from Lemma \ref{case 1} to Lemma \ref{case 8}.
\begin{flushright}
$\Box$
\end{flushright}

Now we prove our second main theorem.
\begin{thm}
The subgroup $T=\langle t_v ~|~ v \in \{1,2,3\}^* \rangle$ in $G_3$ has finite decomposition complexity with respect to any proper length metric. More precisely, $T\in \mathcal{D}_\omega$
\end{thm}

\begin{proof}
By section 2.5, we can take any proper length function on $T$. Define a proper length function $l$ on the generating set $J=\{t_v ~|~v \in \{1,2,3\}^*\}$ of $T$ by $l(t_v^{\pm 1})=|v|$.
$l$ can be extended to a length function on $T$ by the following formula:
$$l(g)=\mbox{min} \big\{~ \sum_{i=1}^n l(t_{v_i}) ~\big|~ g=t_{v_1}^{\pm1} t_{v_2}^{\pm1}\cdots t_{v_n}^{\pm1},n\in \mathbb{N},t_{v_i}\in J   ~\big\},$$
where $g\in T$. It's easy to check that $l$ is proper on $T$.

For any $n\in \mathbb{N}\cup \{0\}$, define $T_n=\langle t_v ~|~ v \in \{1,2,3\}^* \mbox{~and~} |v|\leqslant n \rangle$.
For any mutually different right cosets $T_ng$ and $T_nh$ in $T$, we have $gh^{-1}\notin T_n$, so $l(gh^{-1})\geqslant n+1$. In fact, we can prove $l(k)\geqslant n+1$ for any $k \notin T_n$. To see this, take any minimal representation of $k$: $k=t_{v_1}^{\pm1} t_{v_2}^{\pm1}\cdots t_{v_n}^{\pm1}$ for some $t_{v_i}\in J$. Because $k \notin T_n$, there exists some generator $t_{v_i}$ with $|v_i|\geqslant n+1$. By the definition of $l$, we see $l(k) \geqslant n+1$.
Now
$$d(T_ng,T_nh)=d(T_ngh^{-1},T_n)=\mbox{min}_{u,v\in T_n}l(ugh^{-1}v^{-1})\geqslant n+1.$$
For the last inequality, notice $ugh^{-1}v^{-1} \notin T_n$.
So we only need to check $T_n$ has FDC for all $n$. More precisely, we claim that $T_n\in \mathcal{D}_{4^{n+1}}$. If it's true, for any $R>0$, take some $n>R$, then $T$ has a decomposition into all of the right cosets of $T_n$:
$$T=\bigsqcup_{g} T_ng,$$
and the distance between different cosets are greater than $R$. By the definition of FDC, we know $T\in \mathcal{D}_\omega$.
We prove the claim by induction on $n$.

$T_0=\langle t \rangle\cong\mathbb{Z}\in\mathcal{D}_1$, also contained in $\mathcal{D}_4$. Suppose $T_{n-1}$ belongs to $\mathcal{D}_{4^n}$.
There are three natural ways to embed $T_{n-1}$ into $T_n$. Define $j_1,j_2,j_3:T_{n-1} \rightarrow T_n$ induced by $j_1(t_v)=t_{1v}$, $j_2(t_v)=t_{2v}$ and $j_3(t_v)=t_{3v}$. It's easy to check these three maps are well defined and injective.
And it's also a straightforward calculation that their images $j_1(T_{n-1})$, $j_2(T_{n-1})$ and $j_3(T_{n-1})$ are commutative with each other.

We claim that $j_1(T_{n-1})\oplus j_2(T_{n-1})\oplus j_3(T_{n-1})$ is normal in $T_n$. In fact, notice that
\begin{equation}\label{T_n}
T_n=\langle j_1(T_{n-1})\oplus j_2(T_{n-1})\oplus j_3(T_{n-1}),t  \rangle.
\end{equation}
So we only need to check that $t^{-1}\cdot t_v \cdot t \in j_1(T_{n-1})\oplus j_2(T_{n-1})\oplus j_3(T_{n-1})$ for any finite word $v$ with $1\leqslant |v| \leqslant n$. From Proposition \ref{comm relation}, it's obvious.

Finally, from equation (\ref{T_n}), we know $T_n/\big(j_1(T_{n-1})\oplus j_2(T_{n-1})\oplus j_3(T_{n-1})\big)\cong\langle  \bar{t} \rangle$, which is isomorphic to $\mathbb{Z}$ or $\mathbb{Z}/n\mathbb{Z}$. From Proposition \ref{extension} which tells us that FDC is preserved by extension, we know that $T_n$ has FDC by the assumption on $T_{n-1}$. Since $T_{n-1}\in \mathcal{D}_{4^n}$, we have $j_1(T_{n-1})\oplus j_2(T_{n-1})\oplus j_3(T_{n-1}) \in \mathcal{D}_{3\cdot4^n}$, so $T_n\in \mathcal{D}_{3\cdot4^n+1} \subseteq \mathcal{D}_{4^{n+1}}$.
\end{proof}

Finally, we present an unsolved problem concerning FDC of all mother groups.

\noindent {\bf Problem:}
Do mother groups of bounded automata groups have FDC? In particular, does the Grigorchuk group $\mathbb{G}$ have FDC?

\section*{References}

\bibliography{mybibfile}

\begin{thebibliography}{10}
\expandafter\ifx\csname url\endcsname\relax
  \def\url#1{\texttt{#1}}\fi
\expandafter\ifx\csname urlprefix\endcsname\relax\def\urlprefix{URL }\fi
\expandafter\ifx\csname href\endcsname\relax
  \def\href#1#2{#2} \def\path#1{#1}\fi

\bibitem{Glu61}
V.~M. Glu\v{s}kov, The abstract theory of automata, Russian Mathematical
  Surveys 16~(5) (1961) 1--53.

\bibitem{Gri06}
R.~I. Grigorchuk, I.~Pak, Groups of intermediate growth: an introduction for
  beginners, arXiv: math.GR/0607384.

\bibitem{Gri85}
R.~I. Grigorchuk, Degrees of growth of finitely generated groups and the theory
  of invariant means, Math. USSR-Izv. 25 (1985) 259--300.

\bibitem{Gri80}
R.~I. Grigorchuk, On \textsc{B}urnside problem on periodic groups, Funct. Anal.
  Appl. 14~(1) (1980) 41--43.

\bibitem{Sid00}
S.~Sidki, Automorphisms of one-rooted trees: growth, circuit structure, and
  acyclicity, Journal of Mathematical Sciences 100~(1) (2000) 1925--1943.

\bibitem{Sid04}
S.~Sidki, Finite automata of polynomial growth do not generate a free group,
  Geometriae Dedicata 108~(1) (2004) 193--204.

\bibitem{GS83}
N.~Gupta, S.~Sidki, On the \textsc{B}urnside problem for periodic groups,
  Mathematische Zeitschrift 182 (1983) 385--388.

\bibitem{BKN10}
L.~Bartholdi, V.~A. Kaimanovich, V.~V. Nekrashevych, On amenability of automata
  groups, Duke Mathematical Journal 154~(3) (2010) 575--598.

\bibitem{Yu98}
G.~Yu, The \textsc{N}ovikov conjecture for groups with finite asymptotic
  dimension, The Annals of Mathematics 147~(2) (1998) 325--355.

\bibitem{Smi07}
J.~Smith, The asymptotic dimension of the first \textsc{G}rigorchuk group is
  infinity, Revista Matemática Complutense 20~(1) (2007) 119--121.

\bibitem{GTY12}
E.~Guentner, R.~Tessera, G.~Yu, A notion of geometric complexity and its
  application to topological rigidity, Inventiones mathematicae 189~(2) (2012)
  315--357.

\bibitem{Nek05}
V.~Nekrashevych, Self-similar groups, Vol. 117 of Math. Surveys Monogr.,
  American Mathematical Soc., 2005.

\bibitem{BD05}
G.~Bell, A.~Dranishnikov, Asymptotic dimension in \textsc{B}\c{e}dlewo,
  Topology Proc 38 (2011) 209--236.

\bibitem{WC11}
Y.~Wu, X.~Chen, On finite decomposition complexity of \textsc{T}hompson group,
  Journal of Functional Analysis 261~(4) (2011) 981--998.

\bibitem{NY12}
P.~W. Nowak, G.~Yu, Large scale geometry, 2012.

\bibitem{Har00}
P.~de~La~Harpe, Topics in geometric group theory, University of Chicago Press,
  2000.

\end{thebibliography}

\end{document}